\documentclass[a4paper,12pt]{amsart}

\usepackage[none]{hyphenat}
\usepackage{a4wide,hyperref,amsmath,amssymb,enumitem, stmaryrd}
\usepackage{tikz}
\usetikzlibrary{decorations.pathreplacing,calligraphy}
\usetikzlibrary{shapes}
\tikzset{
			inner sep=1pt,semithick,
			vertex/.style={circle,draw,fill,minimum size= 1pt},
			vertexb/.style={rectangle,draw,fill, minimum size = 3pt},
			vertexd/.style={rectangle,draw,fill=white, minimum size = 3pt},			
			vertexc/.style={circle,draw,fill=white,minimum size= 1pt},			
			thickedge/.style={line width=1pt},				
			font=\tiny
}

\newtheorem{problem}{Problem}
\newtheorem{claim}{Claim}

\newcommand{\lcr}{\overline{\operatorname{cr}}}
\newcommand{\Cr}{\operatorname{cr}}

\newcommand{\old}[1]{{}}
\makeatletter
\DeclareRobustCommand{\cev}[1]{%
  {\mathpalette\do@cev{#1}}%
}
\newcommand{\do@cev}[2]{%
  \vbox{\offinterlineskip
    \sbox\z@{$\m@th#1 x$}%
    \ialign{##\cr
      \hidewidth\reflectbox{$\m@th#1\vec{}\mkern4mu$}\hidewidth\cr
      \noalign{\kern-\ht\z@}
      $\m@th#1#2$\cr
    }%
  }%
}
\makeatother

\title[Curious crossing-critical edges]{Curious crossing-critical edges}

\author[\'E. Czabarka, A. Helm]{\'Eva Czabarka, Alec Helm}
\address{\'Eva Czabarka and Alec Helm\\ Department of Mathematics \\ University of South Carolina \\ Columbia, SC 29208 \\ USA}
\email{czabarka@math.sc.edu,ah191@email.sc.edu }

\begin{document}

\begin{abstract}
Motivated by Kuratowski's theorem, a Kuratowski subgraph of a graph is a subgraph that is a subdivided $K_5$ or a subdivided $K_{3,3}$. An edge is crossing-critical if the crossing number decreases after removing the edge.
In this note, we present the following examples: a graph with an edge that is crossed in every optimal drawing of the graph, but the edge is not in any Kuratowski subgraph of the graph; a graph with a unique edge that is in every Kuratowski subgraph but is not crossed in any optimal drawing of the graph; and
a graph with a crossing-critical edge that is not present in any Kuratowski subgraph and is not crossed in any optimal drawing of the graph. 
F\'ary's theorem implies that the Kuratowski subgraphs are the only obstructions to a graph having a crossing-free drawing with all edges drawn as straight lines. The three example graphs given also hold if we restrict drawings to only have straight line edges, and thus also apply to the rectilinear crossing number.
\end{abstract}

\maketitle

\section{Introduction}

In this note, a graph is a finite loopless multigraph, i.e., we allow parallel edges but no loops.  Edges are  parallel if they have the same endvertices.
Given vertices $x,y$ of the graph $G$,  the \emph{multiplicity} of the $x$-$y$ edge is the number of $x$-$y$ edges in $G$.
A \emph{simple graph} is a graph $G$ without parallel edges, i.e., a graph where the edge multiplicity is one for each adjacent vertex pair $x,y$.
Given a multigraph $G$, the \emph{underlying simple graph} is a simple graph $H$ on the same vertex set, 
such that $H$ has a single edge between  vertices $x,y$ precisely when $x,y$ are adjacent in $G$. 
If $G$ is a graph, a \emph{representation} of $G$ is the pair $(H,w_G)$, where $H$ is the underlying simple graph, and $w_G:E(H)\rightarrow \mathbb{Z}^+$ is defined such that for any $x$-$y$ edge $e$ of $H$, $w_G(e)$ is the edge multiplicity of the $x$-$y$ edge in $G$. 
If $G$ is a graph with underlying simple graph $H$ and $e$ is an $x$-$y$ edge of $H$, then the $x$-$y$-edges of $G$ are \emph{copies} of $e$.
 
A \emph{subdivision of a graph $H$} is a graph $G$ where some of the edges of $H$ are replaced by paths whose internal vertices are not in $H$  and not shared with any other path.
The vertices of $H$ are called the \emph{main vertices} of the subdivision, and the internal vertices of the paths are  called the \emph{subdividing vertices}.

A \emph{drawing $\mathcal{D}$ of a graph $G$} is an embedding of the graph in the plane such that the vertices are assigned to points of the plane injectively, the edges
are mapped to Jordan arcs connecting to the endvertices of the edge, and if $v$ is not an endvertex of the edge $e$ then the image  $\mathcal{D}(v)$ of $v$ is disjoint from the image 
$\mathcal{D}(e)$ of $e$.

A graph is \emph{planar} if it can be drawn in the plane without crossing edges, i.e., such that the images of two edges only have their common endvertices in common; in other words edges do not cross in the drawing.
The celebrated Kuratowski theorem \cite{kuratowski} states that a graph if planar precisely when it does not contain a subdivision of the complete graph $K_5$ or the complete bipartite
graph $K_{3,3}$. $H$ is called  a \emph{Kuratowski subgraph} of the graph $G$ if $H$ is a subgraph of $G$ and $H$ is a subdivision of either $K_{3,3}$ or $K_5$. Note that Kuratowski subgraphs  are simple graphs  by definition.
In a sense, Kuratowski's theorem states that Kuratowski subgraphs force crossings to appear in 
 any drawing of the graph. F\'ary's theorem (proved independently in \cite{fary, wagfary}) states that every planar graph can be drawn without crossings even if we restrict to drawings using straight line edges (the \emph{rectilinear drawings}). Thus, the Kuratowski subgraphs are also the only obstacles to having planar straight line drawings. 
 Moreover, in any drawing of a graph, every Kuratowski subgraph must have a pair of edges that cross.

The \emph{crossing number $\Cr(\mathcal{D})$} of a drawing $\mathcal{D}$ of $G$ is $\sum\limits_{e,f\in E(G)}|\mathcal{D}(e)^0\cap\mathcal{D}(f)^0|$, where for an edge $e$, 
$\mathcal{D}(e)^0$ is the interior of
$\mathcal{D}(e)$ , and the summation goes for unordered pairs of distinct edges. The \emph{crossing number $\Cr(G)$ of a graph $G$} is the minimum crossing number over all its drawings. A drawing $\mathcal{D}$ 
of $G$ is  called \emph{optimal} if $\Cr(\mathcal{D})=\Cr(G)$. An edge $e\in E(G)$ is called \emph{crossing-critical} if $\Cr(G-e)<\Cr(G)$.
The \emph{rectilinear crossing number} $\lcr(G)$ of the graph $G$ is the minimum crossing number of its rectilinear drawings. Clearly $\Cr(G)\le\lcr(G)$, with equality precisely when $G$ has an optimal drawing that is rectilinear.
An edge $e\in E(G)$ is called \emph{crossing-critical with respect to rectilinear drawings} if $\lcr(G-e)<\lcr(G)$.

We call an edge $f$ of a graph $G$ a \emph{Kuratowski edge} if it belongs to some Kuratowski subgraph of $G$, and a Kuratowski edge $f$ is further called a \emph{strong Kuratowski edge} if it belongs to every Kuratowski subgraph. The Kuratowski theorem implies that $f$ is a strong Kuratowski edge of the graph $G$ precisely when $G$ is nonplanar but $G-f$ is planar. In particular, strong Kuratowski edges are both crossing-critical and crossing-critical with respect to rectilinear drawings.

If $h$ is an edge that is crossed in some optimal drawing $\mathcal{D}$ of $G$, then $h$ is clearly crossing-critical, as the drawing of $G-h$ induced by $\mathcal{D}$ has fewer than $\Cr(G)$ crossings; similar statement applies to the rectilinear versions of the concept.

\begin{figure}[http]
\center
\begin{tikzpicture}[scale=1]
\node[draw,circle, minimum size = 13pt] (u) at (0,0) {$u$};
\node[draw,circle, minimum size = 13pt] (v1) at (3,0) {$v_1$};
\node[draw,circle, minimum size = 13pt] (v3) at (.5,1) {$v_3$};
\node[draw,circle, minimum size = 13pt] (v2) at (2,.5) {$v_2$};
\node[draw,circle, minimum size = 13pt] (t) at (1.4,2.3) {$t$};
\node[draw,circle, minimum size = 13pt] (v4) at (0,3) {$v_4$};
\node[draw,circle, minimum size = 13pt] (s) at (3,3) {$s$}; 
\draw (u) to [out=-45,in=225] (v1);
\draw (u) to [out=-35,in=215] (v1);
\draw (u) to [out=-25,in=205] (v1);
\draw (s) to [out=-45,in=45] (v1);
\draw (s) to [out=-55,in=55] (v1);
\draw (s) to [out=-65,in=65] (v1);
\draw (v4) to [out=55,in=-235] (s);
\draw (v4) to [out=45,in=-225] (s);
\draw (v4) to [out=35,in=-215] (s);
\draw (v4) to [out=25,in=-205] (s);
\draw (u) to [out=145,in=-145] (v4);
\draw (u) to [out=135,in=-135] (v4);
\draw (u) to [out=125,in=-125] (v4);
\draw (u) to [out=115,in=-115] (v4);
\draw (u) to [out=105,in=-105] (v4);
\draw (u) to [out=95,in=-95] (v4);
\draw (v4) to [out=-10,in=160] (t);
\draw (v4) to [out=-20,in=170] (t);
\draw (u) to [out=85,in=-135] (v3);
\draw (u) to [out=75,in=-125] (v3);
\draw (u) to [out=65,in=-115] (v3);
\draw (u) to [out=45,in=-165] (v2);
\draw (u) to [out=35,in=-155] (v2);
\draw (u) to [out=25,in=-145] (v2);
\draw (u) to [out=15,in=-135] (v2);
\draw (u) to [out=5,in=-125] (v2);
\draw (u) to [out=-5,in=-115] (v2);
\draw (t) to [out=-95,in=125] (v2);
\draw (t) to [out=-85,in=115] (v2);
\draw (t) to [out=-75,in=95] (v2);
\draw (t) to [out=-125,in=85] (v3);
\draw (t) to [out=-105,in=65] (v3);
\draw (s) to [out=225,in=65] (v2);
\draw (s) to [out=235,in=55] (v2);
\draw (s) to [out=245,in=45] (v2);
\draw (v1) to [out=100,in=0] (1.3,2.8) to [out=180, in=90] (.15,1.7) to [out=270,in=135] (v3);
\node at (2.3,2.9) {$e$};
\draw (s)--(t);
\node at (1.5,-1.2) {An optimal drawing of};
\node at (1.5,-1.6) {the graph $G^*$.};
\end{tikzpicture}
\quad
\begin{tikzpicture}[scale=1]
\node[draw,circle, minimum size = 13pt] (u) at (0,0) {$u$};
\node[draw,circle, minimum size = 13pt] (v1) at (3,0) {$v_1$};
\node[draw,circle, minimum size = 13pt] (v3) at (.5,1) {$v_3$};
\node[draw,circle, minimum size = 13pt] (v2) at (2,.5) {$v_2$};
\node[draw,circle, minimum size = 13pt] (t) at (1.4,2.3) {$t$};
\node[draw,circle, minimum size = 13pt] (v4) at (0,3.2) {$v_4$};
\node[draw,circle, minimum size = 13pt] (s) at (3,3.2) {$s$}; 
\draw (u) to [out=-45,in=225] (v1);
\draw (s) to [out=-65,in=65] (v1);
\draw (v4) to [out=35,in=-215] (s);
\draw (u) to [out=115,in=-115] (v4);
\draw (v4) -- (t);
\draw (u) -- (v3);
\draw (u) -- (v2);
\draw (t) -- (v2);
\draw (t) -- (v3);
\draw (s) -- (v2);
\draw (v1) to [out=100,in=0] (1.3,2.8) to [out=180, in=90] (.15,1.7) to [out=270,in=135] (v3);
\node at (2.3,3) {$e$};
\node at (1.4,3) {$f$};
\node at (.6,3.05) {$g$};
\node at (2.55,2.6) {$h$};
\node at (.05,0.5) {$q$};
\node at (1.5,-.6) {$r$};
\draw (s)--(t);
\node at (1.5,-1.2) {The underlying simple graph};
\node at (1.5,-1.6) {$G_0$ of $G^*$.};
\end{tikzpicture}
\quad
\begin{tikzpicture}[scale=1]
\node[draw,circle, minimum size = 13pt] (u) at (0,0) {$u$};
\node[draw,circle, minimum size = 13pt] (v1) at (3,0) {$v_1$};
\node[draw,circle, minimum size = 13pt] (v3) at (.5,1) {$v_3$};
\node[draw,circle, minimum size = 13pt] (v2) at (2,.5) {$v_2$};
\node[draw,circle, minimum size = 13pt] (t) at (1.4,2.3) {$t$};
\node[draw,circle, minimum size = 13pt] (v4) at (0,3.2) {$v_4$};
\node[draw,circle, minimum size = 13pt] (s) at (3,3.2) {$s$}; 
\draw (u) to [out=-45,in=225] (v1);
\draw (s) to [out=-65,in=65] (v1);
\draw (v4) to [out=35,in=-215] (s);
\draw (u) to [out=115,in=-115] (v4);
\draw (v4) -- (t);
\draw (u) -- (v3);
\draw (u) -- (v2);
\draw (t) -- (v2);
\draw (t) -- (v3);
\draw (s) -- (v2);
\draw (v1) to [out=100,in=0] (1.3,2.8) to [out=180, in=90] (.15,1.7) to [out=270,in=135] (v3);
\node at (2.3,3) {$1$}; 
\node at (1.4,3) {$1$}; 
\node at (.6,3.05) {$2$}; 
\node at (2.6,2.6) {$3$};  
\node at (1,1.5) {$2$}; 
\node at (1.5,1.5) {$3$}; 
\node at (.05,0.55) {$3$}; 
\node at (1.5,-.6) {$3$}; 
\node at (3.3,1.6) {$3$}; 
\node at (1.5,3.6) {$4$}; 
\node at (-.3,1.6) {$6$}; 
\node at (1,.4) {$6$}; 
\draw (s)--(t);
\node at (1.5,-1.6) {A representation of $G^*$.};
\end{tikzpicture}
\caption{\u{S}ir\'a\u{n}'s example \cite{siran}: $\Cr(G^*-e)<\Cr(G^*)=6$, $e$ is not in any Kuratowski subgraph of $G^*$, but is crossed in the optimal drawing shown.} 
\label{fig:siran}
\end{figure}

In 1983, \u{S}ir\'a\u{n} \cite{siran} gave an example of a graph $G^*$ that has a crossing-critical edge $e$ such that $e$ is not an edge of any of the  Kuratowski subgraphs of $G^*$. 
 Note, that, as $\Cr(G^*)=6$, $G^*$ contains a $6$-crossing-critical graph $H^*$. Moreover, $H^*$ must contain the edge $e$ that is still a non-Kuratowski edge in $H^*$. Therefore this example shows the existence of a $8$-crossing-critical graph with non-Kuratowski edges.
Figure~\ref{fig:siran}  shows an optimal drawing of this graph  $G^*$, in which $e$ is crossed, its underlying simple graph $G_0$ and a representation of $G^*$.
The edge $f$ on this figure is clearly a strong Kuratowski edge, as its removal renders the graph planar, and it is crossed in the optimal drawing of Figure~\ref{fig:siran}. Figure~\ref{fig:optdrawings} presents two more optimal drawings of this graph, and illustrates that no edge of $G^*$ is crossed in every optimal drawing.

\begin{figure}[http]
\center
\begin{tikzpicture}[scale=1,font=\tiny]
\node[draw,circle, fill=white, minimum size = 13pt] (u) at (0,0) {$u$};
\node[draw,circle, fill=white, minimum size = 13pt] (v1) at (-1,3) {$v_1$};
\node[draw,circle, fill=white, minimum size = 13pt] (v3) at (.5,1) {$v_3$};
\node[draw,circle, fill=white, minimum size = 13pt] (v2) at (2,.5) {$v_2$};
\node[draw,circle, fill=white, minimum size = 13pt] (t) at (1.4,2.3) {$t$};
\node[draw,circle, fill=white, minimum size = 13pt] (v4) at (0,3) {$v_4$};
\node[draw,circle, fill=white, minimum size = 13pt] (s) at (3,3) {$s$}; 
\draw (u) to [out=155,in=265] (v1);
\draw (u) to [out=165,in=255] (v1);
\draw (u) to [out=175,in=245] (v1);
\draw (s) to [out=105,in=55] (v1);
\draw (s) to [out=95,in=65] (v1);
\draw (s) to [out=85,in=75] (v1);
\draw (v4) to [out=55,in=-235] (s);
\draw (v4) to [out=45,in=-225] (s);
\draw (v4) to [out=35,in=-215] (s);
\draw (v4) to [out=25,in=-205] (s);
\draw (u) to [out=145,in=-145] (v4);
\draw (u) to [out=135,in=-135] (v4);
\draw (u) to [out=125,in=-125] (v4);
\draw (u) to [out=115,in=-115] (v4);
\draw (u) to [out=105,in=-105] (v4);
\draw (u) to [out=95,in=-95] (v4);
\draw (v4) to [out=-10,in=160] (t);
\draw (v4) to [out=-20,in=170] (t);
\draw (u) to [out=85,in=-135] (v3);
\draw (u) to [out=75,in=-125] (v3);
\draw (u) to [out=65,in=-115] (v3);
\draw (u) to [out=45,in=-165] (v2);
\draw (u) to [out=35,in=-155] (v2);
\draw (u) to [out=25,in=-145] (v2);
\draw (u) to [out=15,in=-135] (v2);
\draw (u) to [out=5,in=-125] (v2);
\draw (u) to [out=-5,in=-115] (v2);
\draw (t) to [out=-95,in=125] (v2);
\draw (t) to [out=-85,in=115] (v2);
\draw (t) to [out=-75,in=95] (v2);
\draw (t) to [out=-125,in=85] (v3);
\draw (t) to [out=-105,in=65] (v3);
\draw (s) to [out=225,in=65] (v2);
\draw (s) to [out=235,in=55] (v2);
\draw (s) to [out=245,in=45] (v2);
\draw (v1) -- (v3);
\node at (2.3,2.8) {$e$};
\node at (.3,1.6) {$f$};
\draw (s)--(t);
\end{tikzpicture} 
\quad
\begin{tikzpicture}[scale=1,font=\tiny]
\node[draw,circle, fill=white, minimum size = 13pt] (u) at (0,0) {$u$};
\node[draw,circle, fill=white, minimum size = 13pt] (v1) at (3,0) {$v_1$};
\node[draw,circle, fill=white, minimum size = 13pt] (v3) at (2,.5) {$v_3$};
\node[draw,circle, fill=white, minimum size = 13pt] (v2) at (.5,1) {$v_2$};
\node[draw,circle, fill=white, minimum size = 13pt] (t) at (1.4,2) {$t$};
\node[draw,circle, fill=white, minimum size = 13pt] (v4) at (0,3) {$v_4$};
\node[draw,circle, fill=white, minimum size = 13pt] (s) at (3,3) {$s$}; 
\draw (u) to [out=-45,in=225] (v1);
\draw (u) to [out=-35,in=215] (v1);
\draw (u) to [out=-25,in=205] (v1);
\draw (s) to [out=-45,in=45] (v1);
\draw (s) to [out=-55,in=55] (v1);
\draw (s) to [out=-65,in=65] (v1);
\draw (v4) to [out=55,in=-235] (s);
\draw (v4) to [out=45,in=-225] (s);
\draw (v4) to [out=35,in=-215] (s);
\draw (v4) to [out=25,in=-205] (s);
\draw (u) to [out=145,in=-145] (v4);
\draw (u) to [out=135,in=-135] (v4);
\draw (u) to [out=125,in=-125] (v4);
\draw (u) to [out=115,in=-115] (v4);
\draw (u) to [out=105,in=-105] (v4);
\draw (u) to [out=95,in=-95] (v4);
\draw (v4) to [out=-10,in=160] (t);
\draw (v4) to [out=-20,in=170] (t);
\draw (u) to [out=25,in=-135] (v3);
\draw (u) to [out=15,in=-125] (v3);
\draw (u) to [out=5,in=-115] (v3);
\draw (u) to [out=85,in=-165] (v2);
\draw (u) to [out=75,in=-155] (v2);
\draw (u) to [out=65,in=-145] (v2);
\draw (u) to [out=55,in=-135] (v2);
\draw (u) to [out=45,in=-125] (v2);
\draw (u) to [out=35,in=-115] (v2);
\draw (t) to [out=-125,in=75] (v2);
\draw (t) to [out=-115,in=65] (v2);
\draw (t) to [out=-105,in=55] (v2);
\draw (t) to [out=-85,in=85] (v3);
\draw (t) to [out=-75,in=65] (v3);
\draw (s) to [out=-95,in=0] (v2);
\draw (s) to [out=175,in=115] (v2);
\draw (s) to [out=185,in=105] (v2);
\draw (v1) -- (v3);
\node at (2.1,2.6) {$e$};
\node at (2.6,.4) {$f$};
\draw (s)--(t);
\end{tikzpicture} 
\caption{Optimal drawings of the graph $G^*$ in  \u{S}ir\'a\u{n}'s example \cite{siran}. Note that the set of edges crossed in both of these drawings is empty.} 
\label{fig:optdrawings}
\end{figure}

Based on \u{S}ir\'a\u{n}'s example, we consider the following
\begin{problem}\label{pr:main}   Find the answer to the following questions:
\begin{enumerate}[label={\upshape (\arabic*)}]
\item If an edge $e$ of a graph $G$ is crossed in every optimal drawing of $G$, must $e$ be a Kuratowski edge?
\item If $G$ has strong Kuratowski edges, must $G$ have an optimal drawing in which some strong Kuratowski edge is crossed?
\item If the edge $e$ of a graph $G$ is not crossed in any optimal drawing of $G$ and is not a Kuratowski edge, is it true that $e$ cannot be crossing-critical?
\end{enumerate}
\end{problem}
We consider the corresponding questions for the rectilinear drawings of simple graphs as well.

We modify \u{S}ir\'a\u{n}'s example by changing the edge-multiplicities of $G^*$ judiciously to create graphs $G_1,G_2,G_3$ (see Figure~\ref{fig:siranmod}) such that each have the same underlying simple graph $G_0$ that give the following (perhaps surprising) answers to our questions:
\begin{claim}\label{cl:main} The answer to all questions in {\rm{Problem~\ref{pr:main}}} is in the negative (see Figure~\ref{fig:siranmod} for the examples), i.e.,
\begin{enumerate}[label={\upshape(\arabic*)}]
\item A graph can have an edge  that is crossed in every optimal drawing, but is not a Kuratowski sedge. (An example is the graph $G_1$ and  the unique $s$-$t$ edge $e$.)
\item A graph with strong  Kuratowski edges can be such that no strong Kuratowski edge is crossed in any optimal drawing of the graph. (An example is the graph $G_2$ and the unique strong Kuratowski edge $f$.)
\item A graph can have an edge that is crossing-critical, but this edge is not a Kuratowski edge and is not crossed in any optimal drawing of the graph. 
(An example is the graph $G_3$ and the unique $s$-$t$ edge $e$.)
\end{enumerate}
Furthermore, all three negative answers can be provided with simple graphs. Moreover, these simple graphs provide the corresponding examples for the rectilinear crossing number as well.
\end{claim}

We note that the ``furthermore" part of the previous claim is obvious: in each of the $G_i$, subdivide all but one of the parallel copies of each $x$-$y$ edge to obtain a simple graph $G_i^*$.  The edge $e$ does not have parallel copies, so it is never subdivided, and the edge $f$ does not have parallel copies in $G_2$. Hence these simple graphs also have the required properties.
The ``moreover" part follows from the fact that we have straight line drawings of the underlying simple graphs $G_i$ (Figure~\ref{fig:help}) used in the proof, and
the simple graphs $G_i^*$ have corresponding straight line drawing: draw the not subdivided copy of a multi-edge as a straight line, and place the subdivided 
copies of this edges of the multi-edges close to the original edge.

\begin{figure}[http]
\center
\begin{tikzpicture}[scale=.9, font=\tiny]
\node[draw,circle, fill=white, minimum size = 13pt] (u) at (0,0) {$u$};
\node[draw,circle, fill=white, minimum size = 13pt] (v1) at (3,0) {$v_1$};
\node[draw,circle, fill=white, minimum size = 13pt] (v3) at (.5,1) {$v_3$};
\node[draw,circle, fill=white, minimum size = 13pt] (v2) at (2,.5) {$v_2$};
\node[draw,circle, fill=white, minimum size = 13pt] (t) at (1.4,2.3) {$t$};
\node[draw,circle, fill=white, minimum size = 13pt] (v4) at (0,3) {$v_4$};
\node[draw,circle, fill=white, minimum size = 13pt] (s) at (3,3) {$s$}; 
\draw (u) to [out=-45,in= 180] (1.5,-.7) to [out=0,in=225] (v1);
\draw (u) to [out=-55,in= 180] (1.5,-.8) to [out=0,in=235] (v1);
\draw (u) to [out=-65,in= 180] (1.5,-.9) to [out=0,in=245] (v1);
\draw (u) to [out=-75,in= 180] (1.5,-1) to [out=0,in=255] (v1);
\draw (u) to [out=-85,in= 180] (1.5,-1.1) to [out=0,in=265] (v1);
\draw (u) to [out=-95,in= 180] (1.5,-1.2) to [out=0,in=275] (v1);
\draw (u) to [out=-105,in= 180] (1.5,-1.3) to [out=0,in=285] (v1);
\draw (u) to [out=-115,in= 180] (1.5,-1.4) to [out=0,in=295] (v1);
\draw (u) to [out=105,in=-155] (v3);
\draw (u) to [out=95,in=-145] (v3);
\draw (u) to [out=85,in=-135] (v3);
\draw (u) to [out=75,in=-125] (v3);
\draw (u) to [out=65,in=-115] (v3);
\draw (u) to [out=55,in=-105] (v3);
\draw (u) to [out=45,in=-95] (v3);
\draw (u) to [out=35,in=-85] (v3);
\draw (u) to [out=35,in=-165] (v2);
\draw (u) to [out=25,in=-155] (v2);
\draw (u) to [out=15,in=-145] (v2);
\draw (u) to [out=5,in=-135] (v2);
\draw (u) to [out=-5,in=-125] (v2);
\draw (u) to [out=-15,in=-115] (v2);
\draw (u) to [out=-25,in=-105] (v2);
\draw (u) to [out=-35,in=-95] (v2);
\draw (u) to [out=175,in=-175] (v4);
\draw (u) to [out=165,in=-165] (v4);
\draw (u) to [out=155,in=-155] (v4);
\draw (u) to [out=145,in=-145] (v4);
\draw (u) to [out=135,in=-135] (v4);
\draw (u) to [out=125,in=-125] (v4);
\draw (u) to [out=115,in=-115] (v4);
\draw (u) to [out=105,in=-105] (v4);
\draw (s) to [out=-85,in=85] (v1);
\draw (s) to [out=-75,in=75] (v1);
\draw (s) to [out=-65,in=65] (v1);
\draw (s) to [out=-55,in=55] (v1);
\draw (s) to [out=-45,in=45] (v1);
\draw (s) to [out=-35,in=35] (v1);
\draw (s) to [out=-25,in=25] (v1);
\draw (s) to [out=-15,in=15] (v1);
\draw (v4) to [out=45,in= 180] (1.5,3.7) to [out=0,in=-225] (s);
\draw (v4) to [out=55,in= 180] (1.5,3.8) to [out=0,in=-235] (s);
\draw (v4) to [out=65,in= 180] (1.5,3.9) to [out=0,in=-245] (s);
\draw (v4) to [out=75,in= 180] (1.5,4) to [out=0,in=-255] (s);
\draw (v4) to [out=85,in= 180] (1.5,4.1) to [out=0,in=-265] (s);
\draw (v4) to [out=95,in= 180] (1.5,4.2) to [out=0,in=-275] (s);
\draw (v4) to [out=105,in= 180] (1.5,4.3) to [out=0,in=-285] (s);
\draw (v4) to [out=115,in= 180] (1.5,4.4) to [out=0,in=-295] (s);
\draw (v4) to [out=0,in=130] (t);
\draw (v4) to [out=-10,in=140] (t);
\draw (v4) to [out=-20,in=150] (t);
\draw (t) to [out=-95,in=145] (v2);
\draw (t) to [out=-85,in=135] (v2);
\draw (t) to [out=-75,in=125] (v2);
\draw (t) to [out=-65,in=115] (v2);
\draw (t) to [out=-55,in=105] (v2);
\draw (t) to [out=-45,in=95] (v2);
\draw (t) to [out=-35,in=85] (v2);
\draw (t) to [out=-25,in=75] (v2);
\draw (t) to [out=-165,in=125] (v3);
\draw (t) to [out=-155,in=115] (v3);
\draw (t) to [out=-145,in=105] (v3);
\draw (t) to [out=-135,in=95] (v3);
\draw (t) to [out=-125,in=85] (v3);
\draw (t) to [out=-115,in=75] (v3);
\draw (t) to [out=-105,in=65] (v3);
\draw (t) to [out=-95,in=55] (v3);
\draw (s) to [out=225,in=65] (v2);
\draw (s) to [out=235,in=55] (v2);
\draw (s) to [out=245,in=45] (v2);
\draw (v1) to [out=100,in=0] (1.3,2.8) to [out=180, in=90] (.15,1.7) to [out=270,in=135] (v3);
\node at (2.3,2.9) {$e$};
\node at (1.4,3) {$f$};
\draw (s)--(t);
\node at (1.5,-1.7) {\normalsize{The graph $G_1$}};
\node at (1.5,-2.2) {\normalsize{$w_{1}(f)=w_{1}(e)=1$,}};
\node at (1.5,-2.7) {\normalsize{$w_{1}(g)=w_{1}(h)=3$}};
\node at (1.5,-3.2) {\normalsize{$z\notin\{e,f,g,h\}$: $w_{1}(z)=8$}};
\end{tikzpicture} 
\,\,\,
\begin{tikzpicture}[scale=.9, font=\tiny]
\node[draw,circle, fill=white, minimum size = 13pt] (u) at (0,0) {$u$};
\node[draw,circle, fill=white, minimum size = 13pt] (v1) at (3,0) {$v_1$};
\node[draw,circle, fill=white, minimum size = 13pt] (v3) at (.5,1) {$v_3$};
\node[draw,circle, fill=white, minimum size = 13pt] (v2) at (2,.5) {$v_2$};
\node[draw,circle, fill=white, minimum size = 13pt] (t) at (1.4,2.3) {$t$};
\node[draw,circle, fill=white, minimum size = 13pt] (v4) at (0,3) {$v_4$};
\node[draw,circle, fill=white, minimum size = 13pt] (s) at (3,3) {$s$}; 
\draw (u) to [out=-45,in= 180] (1.5,-.7) to [out=0,in=225] (v1);
\draw (u) to [out=-55,in= 180] (1.5,-.8) to [out=0,in=235] (v1);
\draw (u) to [out=-65,in= 180] (1.5,-.9) to [out=0,in=245] (v1);
\draw (u) to [out=-75,in= 180] (1.5,-1) to [out=0,in=255] (v1);
\draw (u) to [out=-85,in= 180] (1.5,-1.1) to [out=0,in=265] (v1);
\draw (u) to [out=95,in=-145] (v3);
\draw (u) to [out=85,in=-135] (v3);
\draw (u) to [out=75,in=-125] (v3);
\draw (u) to [out=65,in=-115] (v3);
\draw (u) to [out=55,in=-105] (v3);
\draw (u) to [out=25,in=-155] (v2);
\draw (u) to [out=15,in=-145] (v2);
\draw (u) to [out=5,in=-135] (v2);
\draw (u) to [out=-5,in=-125] (v2);
\draw (u) to [out=-15,in=-115] (v2);
\draw (u) to [out=145,in=-145] (v4);
\draw (u) to [out=135,in=-135] (v4);
\draw (u) to [out=125,in=-125] (v4);
\draw (u) to [out=115,in=-115] (v4);
\draw (u) to [out=105,in=-105] (v4);
\draw (s) to [out=-85,in=85] (v1);
\draw (s) to [out=-75,in=75] (v1);
\draw (s) to [out=-65,in=65] (v1);
\draw (s) to [out=-55,in=55] (v1);
\draw (s) to [out=-45,in=45] (v1);
\draw (v4) to [out=45,in= 180] (1.5,3.7) to [out=0,in=-225] (s);
\draw (v4) to [out=55,in= 180] (1.5,3.8) to [out=0,in=-235] (s);
\draw (v4) to [out=65,in= 180] (1.5,3.9) to [out=0,in=-245] (s);
\draw (v4) to [out=75,in= 180] (1.5,4) to [out=0,in=-255] (s);
\draw (v4) to [out=85,in= 180] (1.5,4.1) to [out=0,in=-265] (s);
\draw (v4) to [out=-10,in=140] (t);
\draw (v4) to [out=-20,in=150] (t);
\draw (t) to [out=-85,in=135] (v2);
\draw (t) to [out=-75,in=125] (v2);
\draw (t) to [out=-65,in=115] (v2);
\draw (t) to [out=-55,in=105] (v2);
\draw (t) to [out=-45,in=95] (v2);
\draw (t) to [out=-155,in=115] (v3);
\draw (t) to [out=-145,in=105] (v3);
\draw (t) to [out=-135,in=95] (v3);
\draw (t) to [out=-125,in=85] (v3);
\draw (t) to [out=-115,in=75] (v3);
\draw (s) to [out=235,in=55] (v2);
\draw (s) to [out=245,in=45] (v2);
\draw (v1) to [out=100,in=0] (1.3,2.8) to [out=180, in=90] (.15,1.7) to [out=270,in=135] (v3);
\node at (2.3,2.9) {$e$};
\node at (1.4,3) {$f$};
\draw (s)--(t);
\node at (1.5,-1.7) {\normalsize{The graph $G_2$}};
\node at (1.5,-2.2) {\normalsize{$w_2(f)=w_2(e)=1$,}};
\node at (1.5, -2.7) {\normalsize{$w_{2}(g)=w_2(h)=2$}};
\node at (1.5,-3.2) {\normalsize{$z\notin\{e,f,g,h\}$: $w_{2}(z)=5$}};
\end{tikzpicture} 
\,\,\,
\begin{tikzpicture}[scale=.9, font=\tiny]
\node[draw,circle, fill=white, minimum size = 13pt] (u) at (0,0) {$u$};
\node[draw,circle, fill=white, minimum size = 13pt] (v1) at (3,0) {$v_1$};
\node[draw,circle, fill=white, minimum size = 13pt] (v3) at (.5,1) {$v_3$};
\node[draw,circle, fill=white, minimum size = 13pt] (v2) at (2,.5) {$v_2$};
\node[draw,circle, fill=white, minimum size = 13pt] (t) at (1.4,2.3) {$t$};
\node[draw,circle, fill=white, minimum size = 13pt] (v4) at (0,3.2) {$v_4$};
\node[draw,circle, fill=white, minimum size = 13pt] (s) at (3,3.2) {$s$}; 
\draw (u) to [out=-45,in=225] (v1);
\draw (s) to [out=-65,in=65] (v1);
\draw (v4) to [out=35,in=-215] (s);
\draw (u) to [out=115,in=-115] (v4);
\draw (v4) -- (t);
\draw (u) -- (v3);
\draw (u) -- (v2);
\draw (t) -- (v2);
\draw (t) -- (v3);
\draw (s) -- (v2);
\draw (v1) to [out=100,in=0] (1.3,2.8) to [out=180, in=90] (.15,1.7) to [out=270,in=135] (v3);
\node at (2.3,3) {$1$}; 
\node at (1.4,3) {$2$}; 
\node at (.6,3.05) {$3$}; 
\node at (2.6,2.6) {$7$};  
\node at (0.7,1.75) {$22$}; 
\node at (1.4,1.45) {$22$}; 
\node at (.,0.6) {$22$}; 
\node at (1.5,-.6) {$22$}; 
\node at (3.2,1.6) {$22$}; 
\node at (1.5,3.6) {$22$}; 
\node at (-.2,1.6) {$22$}; 
\node at (1,.5) {$22$}; 
\draw (s)--(t);
\node at (1.5,-1.7) {\normalsize{A representation of $G_3$}};
\node at (1.5,-2.2) {\normalsize{$w_{3}(e)=1$, $w_3(f)=2$,}};
\node at (1.5,-2.7) {\normalsize{ $w_{3}(g)=3$, $w_{3}(h)=7$}};
\node at (1.5,-3.2) {\normalsize{$z\notin\{e,f,g,h\}$: $w_{3}(z)=22$}};
\end{tikzpicture}
\caption{The graphs $G_1$ and $G_2$ and a representation of $G_3$. For $i\in\{1,2,3\}$ the representation of $G_i$ is $(G_0,w_i)$.
$e,f,g,h$ are the edges of the underlying graph $G_0$ (see Figure~\ref{fig:siran}).}
\label{fig:siranmod}
\end{figure}

We will show that the graphs $G_1,G_2,G_3$ and the edges $e,f$ satisfy the properties claimed in Claim~\ref{cl:main} by first finding the edges of Kuratowski subgraphs  
of the graphs in
Claim~\ref{cl:kurat}, and then (after some preliminary work) verifying the rest of Claim~\ref{cl:main} in Claim~\ref{cl:drawing}.

\section{Proof of our claims}

First we determine which  edges belong to  Kuratowski subgraphs of graphs with underlying simple graph $G_0$.

\begin{claim}\label{cl:kurat} Let $G$ be a graph with representation $(G_0,w)$.
The non-Kuratowski edges of $G$ are precisely the copies of $e$ in $G$. An edge
$x$ is a strong Kuratowski edge of $G$ precisely when $x$ is a copy of some $y\in E(G_0)\setminus\{e,q,r\}$ with $w(y)=1$.
\end{claim}
\begin{proof}
Let $G$ be a graph with underlying graph $G_0$.
It is clear that the set of Kuratowski edges of $G$ is the set of copies of the Kuratowski edges of $G_0$, and a copy of an edge $y\in E(G_0)$ is a strong Kuratowski edge of $G$ precisely when $w(y)=1$ and $y$ is a strong Kuratowski edge of $G_0$. Therefore we need to show that the set of Kuratowski edges of $G_0$ is $E(G_0)\setminus\{e\}$ and the set of strong Kuratowski edges of $G_0$ is $E(G_0)\setminus\{e,q,r\}$.

$G_0$ has $3$ vertices of degree $4$ ($u$, $s$ and $t$) and $12$ edges.

As any subdivision of $K_5$ has $5$ vertices of degree $4$,  any Kuratowski subgraph of $G_0$ is a subdivided $K_{3,3}$.

Let $X$ be a Kuratowski subgraph of $G$. As $X$ is a subdivided $K_{3,3}$ with $\alpha\ge 0$ subdividing vertices, $X$ has $6+\alpha$ vertices and $9+\alpha$ edges.
As $G_0$ has $7$ vertices and $12$ edges, either $\alpha=0$ and $X$ is obtained from $G_0$ by removing a vertex of degree $3$, or $\alpha=1$ and $X$ is obtained from
$G_0$ by removing $2$ edges.

If $\alpha=0$, then for some $i\in\{1,2,3,4\}$ we have  $X=G_0-v_i$. As $X$ is a bipartite $3$-regular graph, the neighbors of $v_i$ in $G_0$ all have degree $4$, therefore $i\in\{2,4\}$.
However, if $i\in\{2,4\}$ then $X=G_0-v_i$ contains a triangle on $u,v_1,v_3$, contradicting the fact that $X$ bipartite. Therefore $\alpha=1$, and $X$ is obtained from $G_0$ by removing exactly two edges.

As the removal of these two edges results in a subdivision of a $3$-regular graph, we have that the $4$ endvertices of these $2$ edges must be $s,t,u,v_i$ for some $i\in\{1,2,3,4\}$. Therefore, one of the removed edges must have both endpoints in $\{s,t,u\}$. The only such edge is $e$, which is not in any Kuratowski subgraph of $G_0$. So, $X=G-\{e,z\}$ for some edge $z$ incident upon $u$.
Since $X$ must be triangle-free, $z$ is an edge of the triangle on $u,v_1,v_3$, so $x\in\{q,r\}$.
 
 Therefore we have that if $X$ is a Kuratowski subgraph of $G_0$, then $X\in\{G_0-\{e,q\},G_0-\{e,r\}\}$, and it is easy to check that both of these choices result in a Kuratowski subgraph.
 This finishes our proof.
\end{proof}


Next we consider the drawings.
It is well-known \cite{schaefer} that any optimal drawing of a graph must be \emph{good}, i.e., it must be such that any two arcs representing edges in the drawing do not touch each other, they cross at most once, and arcs representing edges that share an endvertex do not cross. 

If $H$ is an  edge-weighted simple graph with edge-weights $w:E(G)\rightarrow\mathbb{R}^+$, \emph{the weighted crossing number $\Cr_w(\mathcal{D})$ of a 
drawing $\mathcal{D}$} is
$$\Cr_w(\mathcal{D})=\sum_{e,f\in E(G)}w(e)\cdot w(f)\cdot|\mathcal{D}(e)^0\cap\mathcal{D}(f)^0|.$$
and $\Cr_w(H)$ is the minimum of the weighted crossing numbers of good drawings of $H$.

We will make use of the following easy claim:

\begin{claim}\label{cl:easy} Let $G$ be any graph with representation $(H,w)$, $\mathcal{D}$ be any optimal drawing of $G$.
The following statements are true:
\begin{enumerate}[label={\upshape (\roman*)}]
\item\label{case:HtoG} For any good drawing $\mathcal{D}_H$ of $H$, there is a drawing $\mathcal{D}_{H,G}$ of $G$ in which two edges of $G$ cross precisely when they are copies of two edges that cross in $\mathcal{D}_H$. In particular
$\Cr(G)\le\Cr_w(\mathcal{D}_H)$.
\item\label{case:same} For any edge $e$ of $H$, its parallel copies in $G$ have the same number of crossings in $\mathcal{D}$.
\item\label{case:nolargecross} If $z$ is an edge of $H$ such that $w(z)>\Cr(G)$ then no copy of $z$ has a crossing in $\mathcal{D}$.
\item\label{case:psum} Let $e\in E(H)$, $Y\subseteq E(H)$ such that $G$ has a copy $e'$ of $e$ such that
in $\mathcal{D}$ the edge $e'$ crosses at least one copy of every edge in $Y$.
Then $w(e)\sum_{y\in Y} w(y)\le\Cr(G)$.
\item\label{case:reduction} $\Cr(G)=\Cr_w(H)$.
\end{enumerate} 
\end{claim}

\begin{proof}
Note that $\mathcal{D}$ is a good drawing.

Proof of~\ref{case:HtoG}: Create a drawing $\mathcal{D}_{H,G}$ from $\mathcal{D}_H$ as follows: for any edge $e$ of $H$ draw the copies of $e$ in $G$ close enough to
the image of $e$ in $\mathcal{D}_H$ so that  copies of two edges that are uncrossed in $\mathcal{D}_H$ do not cross in the drawing created.

Proof of~\ref{case:same}: Assume to the contrary that $e_1,e_2$ are parallel copies of the edge $e$ of $H$ that do not cross the same number of edges in $\mathcal{D}$, without loss of generality $e_1$ has fewer crossings. Redraw $e_2$ close to $e_1$ so $e_2$ crosses the exact same edges as $e_1$. The new drawing of $G$ has fewer crossings that $\mathcal{D}$, a contradiction,

Proof of~\ref{case:nolargecross} follows from~\ref{case:same}.

Proof of~\ref{case:psum}: Let $y_1,y_2,\ldots,y_k$ be an enumeration of the edges of $Y$, and let $y'_i$ be a copy of $y_i$ such that $e'$ and $y'_i$ crosses in $\mathcal{D}$. 
Create the sequence of drawings $\mathcal{D}^0,\mathcal{D}^1,\ldots\mathcal{D}^k$ satisfying $\Cr(\mathcal{D}^i)=\Cr(G)$ as follows:

For $\mathcal{D}^0$ redraw every parallel copy of $e$ different from $e'$ close to $e'$ so they cross the same edges as $e'$. By~\ref{case:same} $\Cr(\mathcal{D}^0)=\Cr(\mathcal{D})=\Cr(G)$.

If for some $i\in[k]$ we have $\mathcal{D}^{i-1}$, create $\mathcal{D}^i$ by redrawing every copy of $y_i$ except $y'_i$ close to $y'_i$. By~\ref{case:same} $\Cr(\mathcal{D}^{i})=\Cr(\mathcal{D}^{i-1})=\Cr(G)$.

In $\mathcal{D}^k$ every copy of $e$ crosses every copy of every edge in $Y$, consequently $w(e)\sum_{y\in Y} w(y)\le\Cr(\mathcal{D}^k)=\Cr(G)$.

To show~\ref{case:reduction}, we apply the same procedure in~\ref{case:psum}:
Let $h_1, h_2,\ldots,h_m$ be an enumeration of the edges of $H$, and $h'_1,\ldots,h'_k$ be edges of $G$ such that $h'_i$ is a fixed copy of $h_i$.
Define the sequence of drawings
$\mathcal{D}^0=\mathcal{D},\mathcal{D}^1,\mathcal{D}^2,\ldots,\mathcal{D}^m$ such that for all $i\in[m]$, $\mathcal{D}^i$ is obtained from $\mathcal{D}^{i-1}$ by redrawing the edges parallel to $h'_i$. As before $\Cr(\mathcal{D}^i)=\Cr(G)$. Let $\mathcal{D}_H$ be the subdrawing of $H$ generated by the edges $h'_1,\ldots,h'_m$ in $\mathcal{D}^m$.
Then $\Cr_w(H)\le\Cr_W(\mathcal{D}_H)=\Cr(\mathcal{D}^m)=\Cr(G)$. Using~\ref{case:HtoG}, the statement follows.
\end{proof}

We detour a bit by defining when two drawings of a graph  are equivalent (see standard graph theory books, e.g.,~\cite{diestel}).  Using $\mathcal{D}(G)$
for the set of points
that are in the drawing, we can define a face of the drawing as a maximal connected component of $\mathbb{R}^2-\mathcal{D}(G)$ (note that this does not assume that the drawing is planar).

Let $S^2$ be the two-dimensional sphere, with north pole $(0,0,1)$, and $\pi:S^2-\{(0,0,1)\}\rightarrow \mathbb{R}^2$ is the stereographic projection of the sphere to the plane.
Let $G$ be a graph and $\mathcal{D}_1,\mathcal{D}_2$ be two drawings of $G$. Drawings $\mathcal{D}_1,\mathcal{D}_2$ are called \emph{equivalent} if there is a homeomorphism
$\phi: S^2\rightarrow S^2$ such that $\pi\circ\phi\circ\pi^{-1}$ maps $\mathcal{D}_1$ to $\mathcal{D}_2$. Note that the use of $\pi, \pi^{-1}$
 allows us to switch the outer face of the drawing without changing any salient features of the drawing.

Recall that for $k\in\mathbb{N}$, a \emph{simple} graph 
is $k$-connected if it has more than $k$ vertices and the removal of fewer than $k$ vertices does not disconnect the graph.
Whitney~\cite{whitney} showed that any two crossing-free drawings of a $3$-connected \emph{simple} planar graph  are equivalent.
The Jordan-Sch\"onflies theorem \cite{mohar2001graphs} states that   if $\gamma, \delta$ are both closed Jordan curves
    in $\mathbb{R}^2$, and $f$ is a homeomorphism
    between them, then $f$ may be extended to a homeomorphism of the entire plane to itself. 
    The analogous statement with $\mathbb{S}^2$ in the place of $\mathbb{R}^2$ also holds. These imply that any two crossing-free drawings of a subdivision of
    $3$-connected simple planar graph  are equivalent.
    
Claim~\ref{cl:easy} allows us to reduce our questions about optimal drawing of a graph into questions about drawings of the underlying simple graph. 
So, we will first consider drawings of the simple graph $G_0$.

\begin{claim}\label{cl:g0} Assume that $\mathcal{D}$ is a drawing of $G_0$ such that the induced subdrawing of $G_0-f$ is planar, and the set of edges that $f$
crosses in $\mathcal{D}$ is a subset of  $\{e,g,h\}$. Then $f$ crosses each of $e$,$g$ and $h$ in $\mathcal{D}$.
\end{claim}
\begin{proof}
Recall the drawing $\mathcal{D}_0$ of $G_0$ (Figure~\ref{fig:help}), and set $H_0=G_0-f$.
$H_0$ is a subdivision of a $K_5$ minus an edge, and $K_5$ minus an edge is a planar $3$-connected simple graph.
Therefore the subdrawings
$\mathcal{D}_0\vert_{H_0},\mathcal{D}\vert_{H_0}$ of $H_0$ induced by the drawings $\mathcal{D}_0,\mathcal{D}$ are equivalent. 
The edges $E(G_0)-\{e,f,g,h\}$ form the closed walk $v_3tv_2uv_1sv_4uv_3$.
The subdrawings of this closed walk in $\mathcal{D}$ (and $\mathcal{D}_0$) divide the plane into the same
$3$ regions, one of which contains the edges
$e,g,h$ (Figure~\ref{fig:help}, left picture).  As $G$ is non-planar, $f$ must cross some edges 
in $\mathcal{D}$ (and $\mathcal{D}_0$).
Since $f$ can only cross the edges of $e,g,h$, edge $f$ must also lie in the same region as  $e,g,h$. 
As the two endvertices of $f$ lie on the $u$-$t$ and $u$-$s$ edges, $f$ crosses all of $e,g,h$.
\end{proof}

\begin{figure}[http]
\center
\begin{tikzpicture}[scale=1]
\node[draw,circle, minimum size = 13pt] (u) at (0,0) {$u$};
\node[draw,circle, minimum size = 13pt] (v1) at  (4,0) {$v_1$};
\node[draw,circle, minimum size = 13pt] (v3) at (.28,2) {$v_3$};
\node[draw,circle, minimum size = 13pt] (v2) at (2.3,.275) {$v_2$};
\node[draw,circle, minimum size = 13pt] (t) at (1.6,.7) {$t$};
\node[draw,circle, minimum size = 13pt] (v4) at (0,3) {$v_4$};
\node[draw,circle, minimum size = 13pt] (s) at (4,3) {$s$}; 
\draw (u) --(v1);
\draw (s) -- (v1);
\draw (v4) -- (s);
\draw (u) -- (v4);
\draw (v4) -- (t);
\draw (u) -- (v3);
\draw (u) -- (v2);
\draw (t) -- (v2);
\draw (t) -- (v3);
\draw (s) -- (v2);
\draw (v1) -- (v3);
\draw (s)--(t);
\node at (3.4,.5) {$f$};
\node at (.8,2.3) {$g$};
\node at (2.6,1.8) {$e$};
\node at (3.4,1.8) {$h$};
\node at (2,-.5) {$\mathcal{D}_0$};
\end{tikzpicture}
\qquad
\begin{tikzpicture}[scale=1]
\node[draw,circle, minimum size = 13pt] (u) at (0,0) {$u$};
\node[draw,circle, minimum size = 13pt] (v1) at  (2.2,0) {$v_1$};
\node[draw,circle, minimum size = 13pt] (v3) at (1,.55) {$v_3$};
\node[draw,circle, minimum size = 13pt] (v2) at (.5,1.3) {$v_2$};
\node[draw,circle, minimum size = 13pt] (t) at (1.6,1.3) {$t$};
\node[draw,circle, minimum size = 13pt] (v4) at (0,3) {$v_4$};
\node[draw,circle, minimum size = 13pt] (s) at (2.2,3) {$s$}; 
\draw (u) --(v1);
\draw (s) -- (v1);
\draw (v4) -- (s);
\draw (u) -- (v4);
\draw (v4) -- (t);
\draw (u) -- (v3);
\draw (u) -- (v2);
\draw (t) -- (v2);
\draw (t) -- (v3);
\draw (s) -- (v2);
\draw (v1) -- (v3);
\draw (s)--(t);
\node at (1.6,.4) {$f$};
\node at (.75,2.4) {$g$};
\node at (1.45,2.4) {$h$};
\node at (1.65,1.8) {$e$};
\node at (1.1,-.5) {$\mathcal{D}_1$};
\end{tikzpicture}
\caption{The (rectilinear) drawings $\mathcal{D}_0$ and $\mathcal{D}_1$ of $G_0$. 
In $\mathcal{D}_0$, the edge $f$ crosses the edges $e,g,h$. In $\mathcal{D}_1$, the edges $g$ and $h$ cross.
} 
\label{fig:help}
\end{figure}

\begin{claim}\label{cl:final}
Let $G$ be a graph with representation $(G_0,w)$ and set $$\kappa=\min\left(w(g)w(h),\,w(f)(w(e)+w(g)+w(h))\right).$$
If for all $z\in E(G_0)\setminus\{e,f,g,h\}$ we have $w(z)>\kappa$, then $\Cr(G)=\kappa$. Moreover
\begin{enumerate}[label={\upshape (\roman*)}]
\item\label{case:ghcross} If $w(g)w(h)<w(f)(w(e)+w(g)+w(h))$, then in any optimal drawing of $G$ only copies of $g$ and $h$ cross. In particular, $e$ and $f$ are never crossed in an optimal drawings.
\item\label{case:gh} If $w(g)w(h)>w(f)(w(e)+w(g)+w(h))$,, then in any optimal drawings of $G$, all copies of $f$ cross all copies of $e,g$ and $h$. In particular, $e$ is crossed in any optimal drawings.
\end{enumerate}
\end{claim}

\begin{proof}
Claim~\ref{cl:easy}~\ref{case:HtoG} and the drawings $\mathcal{D}_0,\mathcal{D}_1$ of $G_0$ on Figure~\ref{fig:help} imply that
$\Cr(G)\le\kappa$. Since for all $z\in E(G_0)\setminus\{e,f,g,h\}$ we have $w(z)>\kappa$, Claim~\ref{cl:easy}~\ref{case:nolargecross} gives that
in any optimal drawing of $G$, only copies of $e,f,g,h$ may cross. Since in optimal drawings incident edges do not cross, if a crossing does not involve a copy of $f$,
it must happen between copies of $g$ and $h$.

Let $\mathcal{D}$ be an optimal drawing of $G$, and let $g',h'$ be a copy of $g$ and $h$ respectively. We have one of following:
\begin{enumerate}[label={\upshape (\alph*)}]
\item\label{case:1} If $g'$ and $h'$ cross in $\mathcal{D}$, then by Claim~\ref{cl:easy}~\ref{case:psum}, $\kappa\le w(g)w(h)\le\Cr(G)\le\kappa$.
\item\label{case:2} If $g'$ and $h'$ do not cross in $\mathcal{D}$, then select a subdrawing of $G_0$ by selecting one copy $z'$ of each edge $z\in E(G_0)$ such that $g',h'$ are selected.
As in this drawing $g',h'$ do not cross, the only crossings involve $f'$ with $e',g',h'$. By Claim~\ref{cl:g0} $f'$ crosses all of $e',g',h'$ and Lemma~\ref{cl:easy}~\ref{case:psum}
gives that $\kappa\le w(f)(w(e)+w(g)+w(h))\le\Cr(G)\le\kappa$.
\end{enumerate}
This gives $\Cr(G)=\kappa$.

If $w(g)w(h)<w(f)(w(e)+w(g)+w(h))$, then by~\ref{case:2} no copy of $f$ may cross any copy of $g,h,e$ in $\mathcal{D}$. Part ~\ref{case:ghcross} follows.

If $w(g)w(h)>w(f)(w(e)+w(g)+w(h))$, then by~\ref{case:1} any copies of $g$ and $h$ do not cross in $\mathcal{D}$. Part ~\ref{case:gh} follows.
\end{proof}

Now we are ready to prove our final claim:

\begin{claim}\label{cl:drawing} The following statements are true:
\begin{enumerate}[label={\upshape (\roman*)}]
\item\label{case:g1} $\Cr(G_1)=7$, and $e$ is crossed in every optimal drawing of $G_1$.
\item\label{case:g2} $\Cr(G_2)=1$, and $f$ is not crossed in any optimal drawing of $G_2$.
\item\label{case:g3} $\Cr(G_3)=21$, and $e$ is not crossed in any optimal drawing of $G_3$, but $\Cr(G-e)<\Cr(G)$.
\end{enumerate}
\end{claim}

\begin{proof}
Recall that for $i\in\{1,2,3\}$, the representation of $G_i$ is $(G_0,w_i)$.
 
 $w_1(h)w_1(g)=9$, $w_1(f)(w_1(e)+w_1(h)+w_1(h))=7$ and $w_1(z)=8$ for $z\notin\{e,f,g,h\}$. Part~\ref{case:g1} follows from Claim~\ref{cl:final}.
 
  $w_2(h)w_2(g)=4$, $w_2(f)(w_2(e)+w_2(h)+w_2(h))=5$ and $w_2(z)=5$ for $z\notin\{e,f,g,h\}$. Part~\ref{case:g2} follows from Claim~\ref{cl:final}.
  
   $w_3(h)w_3(g)=21$, $w_3(f)(w_3(e)+w_3(h)+w_3(h))=22$ and $w_3(z)=22$ for $z\notin\{e,f,g,h\}$.  Claim~\ref{cl:final} implies that
   $\Cr(G_3)=21$, and $e$ is not crossed in any optimal drawing of $G_3$. Using the subdrawing $\mathcal{D}'$ of $G_0-e$ induced by the drawing $\mathcal{D}_0$ (see Figure~\ref{fig:help}), Claim~\ref{cl:easy}~\ref{case:HtoG} gives that $\Cr(G-e)\le\Cr_{w_3}(\mathcal{D}')=20<\Cr(G)$  
\end{proof}

It is clear that Claim~\ref{cl:kurat} and Claim~\ref{cl:drawing} implies Claim~\ref{cl:main}.


\begin{thebibliography}{10}
\bibitem{diestel} R. Diestel,
  \newblock Graph Theory, Sixth edition 2025,
  \newblock Springer-Verlag, Heidelberg
Graduate Texts in Mathematics, Volume 173.

\bibitem{fary} I. F\'ary,
\newblock On straight-line representation of planar graphs,
\newblock  Acta Sci. Math. (Szeged, 1948), 11: 229--233

  \bibitem{kuratowski} K. Kuratowski, 
  \newblock Sur le probl\`eme des courbes gauches en topologie,
  \newblock Fund. Math. (in French) 15 (1930), 271--283. 

\bibitem{mohar2001graphs}
{B. Mohar, C. Thomassen},
 \newblock{Graphs on Surfaces},
  {2001},
  \newblock{Johns Hopkins University Press}

  \bibitem{schaefer} M. Schaefer, 
  \newblock Crossing numbers of graphs,
  \newblock CRC Press, Taylor and Francis group, Discrete Mathematics and its Applications, 2018.

  
  \bibitem{siran} J. \u{S}ir\'a\u{n}, 
  \newblock Crossing-critical edges and Kuratowski subgraphs of a graph,
  \newblock J. Comb. Theory B 35(2) (1983), 83--92. 
\newblock \href{https://doi.org/10.1016/0095-8956(83)90064-3}{https://doi.org/10.1016/0095-8956(83)90064-3}

\bibitem{wagfary}
\newblock K. Wagner
\newblock Bemerkungen zum Vierfarbenproblem
\newblock Jahresbericht der Deutschen Mathematiker-Vereinigung (in German, 1936)  46: 26--32.


  \bibitem{whitney} H. Whitney, 
  \newblock Congruent graphs and the connectivity of graphs,
  \newblock Am. J. Math. 54 (1932) 150--168.
 \end{thebibliography}
\end{document}